\documentclass[12pt]{article}
\usepackage{fullpage}
\usepackage{amssymb,amsmath,amsfonts,amsthm,latexsym,enumerate,url,cases}
\usepackage{amsthm,amsmath,amssymb}
\usepackage{mathrsfs}
\numberwithin{equation}{section}
\usepackage{hyperref}
\usepackage{indentfirst}
\hypersetup{colorlinks=true,citecolor=blue,linkcolor=blue,urlcolor=blue}

\newtheorem{theorem}{Theorem}[section] %
\newtheorem{lemma}[theorem]{Lemma} %
\newtheorem{corollary}[theorem]{Corollary} %
\newtheorem{problem}[theorem]{Problem} %
\newtheorem{remark}[theorem]{Remark} %

\usepackage{cite}
\allowdisplaybreaks[1]
\begin{document}
\title{On arithmetic progressions of positive integers avoiding $p+F_m$ and $q+L_n$}

\author{  Rui-Jing Wang\footnote{ E-mail: wangruijing271@163.com}\\
\small  School of Mathematical Sciences, \\
\small  Jiangsu Second Normal University,  Nanjing, Jiangsu 210013, China\\
}
\date{}
\maketitle \baselineskip 18pt \maketitle \baselineskip 18pt

{\bf Abstract.} In this paper, it is proved that there is an arithmetic progression of positive integers such that each of which is expressible neither as $p+F_m$ nor as $q+L_n$, where $ p,q $ are primes, $ F_m $ denotes the $ m $-th Fibonacci number and $ L_n $ denotes the $ n $-th Lucas number.  \vskip 3mm



%Goldbach-type theorems; other additive questions involving primes
%Fibonacci and Lucas numbers and polynomials and generalizations
%Congruences; primitive roots; residue systems 指的是完全剩余系

\vskip 5mm

\section{Introduction}

%In 1934, Romanoff \cite{Romanoff} proved that there is a positive proportion of all positive integers which can be represented as a sum of a prime and a power of two. In 2004, Chen and Sun \cite{Chen} gave the first quantitative result $0.0868$ for Romanoff's theorem. The best quantitative result $0.107648$ was given by Elsholtz and Schlage-Puchta \cite{Elsholtz}. For improvements of this bound, one may refer to \cite{Habsieger,Lv,Pintz,Sivak,Elsholtz}. 

In 1849, de Polignac\cite{dePolignac1} conjectured that every odd number greater than $ 3 $ can be
expressed as the sum of a prime and a power of two. But soon, he\cite{dePolignac2} found a counterexample $ 959. $ %had already been mentioned by Euler in a letter to Goldbach. 
In 1950, van der Corput\cite{vanderCorput1950} showed that a positive proportion of positive odd numbers cannot be written as the sum of a prime and a power of two. A system of congruences $n\equiv 
a_i\pmod{m_i} \ (1\le i\le t)$ is called a covering system if every integer
$ n $ satisfies at least one of the congruences. In the same year, Erd\H os\cite{Erdos1950} proved that there is an arithmetic progression of positive odd numbers such that each of which cannot be represented as the sum of a prime and a power of two
by using a covering system. 

%After proving that each residue $ r\in \{0,1,\ldots,p-1\} $ occurs at most $ 4 $ times in a period modulo $ p $ for Lucas numbers by the method in \cite{Schinzel}, one may show that there is a positive proportion of all positive integers which can be represented as a sum of a Lucas number and a prime. We can give a lower bound of this proportion by the method of \cite{Wang}.

Let $\mathcal{P}$ be the set of all primes and $\mathcal{F}=\{F_m: m=0,1,2,\ldots\}$ the sequence of all Fibonacci numbers, where $F_0=0$, $F_1=1$ and $F_m=F_{m-1}+F_{m-2}$ $(m\geq 2)$. In 2012, Jones~\cite{Jones2012} proved that there exist infinitely many positive integers such that each of which is expressible neither as $p+f$ nor as $-p+f$, where $p\in \mathcal{P},\ f\in\mathcal{F}.$ In 2014, Ismailescu and Shim\cite{Ismailescu2014} showed that there exist infinitely many integers that cannot
be expressed as $\pm p\pm f$ for any choice of the signs, where $p\in \mathcal{P},\ f\in\mathcal{F}.$ In 2016, \v{S}iurys\cite{Siurys2016} proved that there is an arithmetic progression of positive integers such that each of which cannot be represented as the form $p+f,$ where $p\in \mathcal{P},\ f\in\mathcal{F}.$ Let $\mathcal{L}=\{L_n: n=0,1,2,\ldots\}$ be the sequence of all Lucas numbers, where $L_0=2$, $L_1=1$ and $L_n=L_{n-1}+L_{n-2}$ $(n\ge 2). $ In 2017, Ismailescu\cite{Ismailescu2017} showed that there exist infinitely many integers that cannot
be written as $\pm p\pm l$ for any choice of the signs,  where $p\in \mathcal{P},\ l\in\mathcal{L}.$ 
For more relevant research, one may refer to \cite{ChenDing,ChenDing2,Chen2024,Ding2022,Ding20222,Ding20223,Ding2021,Elsholtz,Li2008,Pan2011,Romanoff,Yuan2004}.

For any positive integer $n$, let 
$$r_f(n)=|\{ (p,m) : n=p+F_m,\ p\in \mathcal{P},\ F_m\in \mathcal{F}\} |,$$ $$r_l(n)=|\{ (p,m) : n=p+L_m,\ p\in \mathcal{P},\ L_m\in \mathcal{L}\} |.$$
%Wang\cite{Wang2024} proved that the set $ \{ n : r_l(n)=0\} $ has a positive lower asymptotic density by a similar covering system. 
Inspired by the ideas of \cite{Jones2012,Ismailescu2014,Ismailescu2017}, in this paper, we consider the set $$ B=\{ n>1 : r_f(n)=r_l(n)=0\}. $$ Based on a covering system given by Chen and Wang\cite{ChenWang2024} in 2024 to prove that both the sets $ \{ n : r_f(n)=1\} $ and $ \{ n : r_f(n)=0\} $ have positive lower asymptotic densities, we construct two covering systems and prove the following result.
\begin{theorem}\label{thm1} Let
	\begin{eqnarray*} M&=&4622887311048560370676603365714605438732035554897681130826120780882327\\
		&& 893281434118684817634570358523156027263412324830,\\
		N&=&3939763412672581121898182599728603280412336895559113342123449529012811\\
		&&376139243610743695185350000191300835935299152713,
	\end{eqnarray*}
	%$$A=\{  M k+N : k=0,1,2, \ldots \} ,$$
	then $ M k+N\in B $ for every integer $k\geq 0.$
\end{theorem}
\begin{remark}Indeed, the proofs in \cite{Jones2012,Ismailescu2014,Ismailescu2017} contain the corresponding arithmetic progressions. We note that our main result cannot be obtained from the intersection of previous results.\end{remark}
\begin{remark} It is clear that
	$$  \{221, 535, 697, 793, 1045, 1075, 1243, 1333, 1347, 1351, 1651, 1961, \ldots\}\subseteq B. $$
	
	By computer programming, we find that there are $ 226844 $ elements of $ B $ less than $ 10^7. $ Although $ M,N $ is inevitablely big due to the method itself, thanks to the exponential growth of $ \mathcal{F} $ and $ \mathcal{L} $, it does not spend much time with the help of a computer on verifying that $ M k+N\in B $ for $0\leq k\leq 10^5 $.\end{remark}
Based on the results of \cite{ChenWang2024}, we give the following problem for further research.
\begin{problem}
	 	Does the set $$ \{ n : r_f(n)=r_l(n)=1\} $$ have a positive lower asymptotic density?
\end{problem}

\section{Preliminaries}

%\begin{definition} \end{definition}

For convenience, we use $\{a_i\pmod{m_i}\}_{i=1}^{t}$ or more simply $\{(a_i,m_i)\}_{i=1}^{t}$ to denote the covering system in the following discussion. An example of covering system is $$\{0\hskip -2.5mm\pmod{2},\quad 0\hskip -2.5mm\pmod{3},\quad 1\hskip -2.5mm\pmod{4},\quad 5\hskip -2.5mm\pmod{6},\quad 7\hskip -2.5mm\pmod{12}\}.$$ The covering system used by Erd\H os\cite{Erdos1950} is $$\{0\hskip -2.5mm\pmod{2},\quad 0\hskip -2.5mm\pmod{3},\quad 1\hskip -2.5mm\pmod{4},\quad 3\hskip -2.5mm\pmod{8},\quad 7\hskip -2.5mm\pmod{12},\quad 23\hskip -2.5mm\pmod{24}\}.$$

\begin{lemma}\label{cs} Let
	\begin{eqnarray*} &&\{ (a_i, m_i ) \}_{i=1}^{33}\\
		&=& \{ (1,3), (2,3),   (3,8), (5,8), (6,8),  (7,16), (9,16), (10,16),(0,48),(24,48),\\
		&& (15, 32),(17, 32),(18, 32), (60,96), (84,96), (9,18),(24,36),(30,36),\\
		&&(33,192),(159,192),
		(1,64),(2,64),(63,64),(12,144),(36,288),(108,288),\\
		&&(546,576),(100,128),(36,128),(324,384),(162,1152),(516,1152),(738,1152)\},
	\end{eqnarray*}
and let
	\begin{eqnarray*} \{ (b_j, n_j) \}_{j=1}^{15}
		&=& \{ (1,3), (2,3), (2,4),  (0,8), (5,8), (7,8),(9,16),(12,48),(36,48),\\
		&&(17, 32), (15,18),(27,36),(3,72),(1,64),(33,192)\} .
	\end{eqnarray*}
Then $\{(a_i, m_i )\}_{i=1}^{33}$ and $\{(b_j, n_j)\}_{j=1}^{15}$ are two covering systems.
\end{lemma}

We note that $$ \text{lcm}(m_1,\ldots,m_{33})=1152,\quad \text{lcm}(n_1,\ldots,n_{15})=576. $$ By calculation, one can easily verify Lemma \ref{cs}.

\section{Proof of Theorem \ref{thm1}}

Our original goal is to construct a set $A$ by covering systems such that if $$n=p+f\ (p\in \mathcal{P},f\in\mathcal{F})\quad \text{or}\quad n=q+l\ (q\in \mathcal{P},l\in\mathcal{L})$$ for any $n\in
A$, then the available choices of primes $ p,q $ can be restricted in a finite subset of $ \mathcal{P} $. For original idea and example, one can refer to \cite{Erdos1950}. 

The numbers $M$ and $N$ come from \eqref{eq2} by the Chinese remainder theorem. For the convenience of the construction, we choose $ 1152=2^{7}\cdot 3^2 $ as the biggest modulo so that the instance in \cite{ChenWang2024} can provide some reference values. 

For any integer $d>1$, let
$$ \chi_f (d)=\min\{ k>0:\
F_{k+1}-F_{1}\equiv F_{k}-F_{0}\equiv 0 \hskip -2.5mm\pmod d \}, $$
$$ \chi_l (d)=\min\{ k>0:\
L_{k+1}-L_{1}\equiv L_{k}-L_{0}\equiv 0 \hskip -2.5mm\pmod d \}. $$
Now that we consider both Fibonacci numbers and Lucas numbers in the construction of Theorem \ref{thm1}, our choices of the following congruences should take a balence between Fibonacci numbers and Lucas numbers, otherwise the biggest modulo $ 1152 $ may not be enough. Actually, we still use all primes $ p $ with $ \chi_f(p)\mid 1152 $ or $ \chi_l(p)\mid 1152 $ after many attempts, which makes $M$ and $ N $ inevitablely big. The feasibility of the discussions after \eqref{eq4f} and \eqref{eq4} usually can be ensured by the properties of Fibonacci numbers and Lucas numbers themselves.

	For any integer $k\geq 0,$ it can be shown with some help of a computer that
	$$F_{3k+1}\equiv F_{3k+2}\equiv L_{3k+1}\equiv L_{3k+2}\equiv 1\hskip -2.5mm\pmod{2}, \quad L_{4k+2}\equiv 3\hskip -2.5mm\pmod{5},$$
	$$ F_{8k+3}\equiv F_{8k+5}\equiv F_{8k+6}\equiv L_{8k}\equiv L_{8k+5}\equiv L_{8k+7}\equiv
	2\hskip -2.5mm\pmod{3}, $$ $$ F_{16k+7}\equiv F_{16k+9}\equiv F_{16k+10}\equiv L_{16k+9}\equiv 6\hskip -2.5mm\pmod{7},$$
	$$F_{48k+0}\equiv
	F_{48k+24}\equiv
	L_{48k+12}\equiv
	L_{48k+36}\equiv
	0\hskip -2.5mm\pmod{23},$$ $$
	F_{32k+15}\equiv F_{32k+17}\equiv F_{32k+18}\equiv L_{32k+17}\equiv 46\hskip -2.5mm\pmod{47},$$ 
	$$ F_{96k+60}\equiv F_{96k+84}\equiv 959\hskip -2.5mm\pmod{1103},\quad F_{18k+9}\equiv L_{18k+15}\equiv 15\hskip -2.5mm\pmod{19},$$ $$ F_{36k+24}\equiv F_{36k+30}\equiv L_{36k+27}\equiv 9\hskip -2.5mm\pmod{17},\quad
	L_{72k+3}\equiv
	4\hskip -2.5mm\pmod{107},$$
	$$ F_{192k+33}\equiv  F_{192k+159}\equiv 2874\hskip -2.5mm\pmod{3167},$$
	$$F_{64k+1}\equiv F_{64k+2}\equiv F_{64k+63}\equiv L_{64k+1}\equiv 1\hskip -2.5mm\pmod{2207},$$ 
	$$ F_{144k+12}\equiv 144\hskip -2.5mm\pmod{103681},\quad L_{192k+33}\equiv 484\hskip -2.5mm\pmod{769},$$ 
	$$ F_{288k+36}\equiv F_{288k+108}\equiv 14930352\hskip -2.5mm\pmod{10749957121},
	$$ 
	$$ F_{576k+546}\equiv 115561578124837690841\hskip -2.5mm\pmod{115561578124838522881},$$ 
	$$  F_{128k+100}\equiv 680\hskip -2.5mm\pmod{1087},\quad F_{128k+36}\equiv 4141\hskip -2.5mm\pmod{4481},$$
	$$ F_{384k+324}\equiv 10314566492783\hskip -2.5mm\pmod{11862575248703},$$ 
	$$F_{1152k+162}\equiv
	65243\hskip -2.5mm\pmod{270143},$$ 
	$$F_{1152k+516}\equiv
	1548008755920\hskip -2.5mm\pmod{25033626656641},$$
	$$F_{1152k+738}\equiv
	784086245571237641757\hskip -2.5mm\pmod{1974737795746080149567}.$$
	\begin{remark}\label{remark} In order to facilitate the interested readers to understand the idea of our construction, the order of congruences is not in accordance with the size of modulos. Since $ \chi_f(5)=20 $ and $ \chi_l(5)=4, $  we often have the opportunity to complete the processing for Lucas numbers first. For example, the following covering system is also enough to prove that there is an arithmetic progression of positive integers such that each of which cannot be represented as the form $p+l$, where $p\in \mathcal{P},\ l\in\mathcal{L}.$ 
	\begin{eqnarray*} 
		&&\{ (1,3), (2,3), (2,4),  (0,8), (5,8), (7,8),(9,16),(12,48),(36,48),\\
		&&(17, 32), (15,18),(27,36),(3,72),(33,96)\}.
	\end{eqnarray*}	
	Moreover, if we only consider the lower asymptotic density of the set $ \{ n : r_l(n)=0\} $, the quantitative result follows from the above covering system will be much better.
	But premature completion of this step is a waste of congruences, resulting in the failure to the later step for Fibonacci numbers.
	
	Besides, we occasionally encounter the same amount of coverage caused by two different choices. For instance,
	$$ F_{8k+1}\equiv F_{8k+2}\equiv F_{8k+7}\equiv L_{8k+1}\equiv L_{8k+3}\equiv L_{8k+4}\equiv
	1\hskip -2.5mm\pmod{3}, $$
	$$ F_{8k+3}\equiv F_{8k+5}\equiv F_{8k+6}\equiv L_{8k}\equiv L_{8k+5}\equiv L_{8k+7}\equiv
	2\hskip -2.5mm\pmod{3}. $$
	After several attempts, we find that the different choices at this time will largely affect the subsequent situation. Choosing the former will often lead to the situation that any choice available can only work for either Fibonacci numbers or Lucas numbers, which is also a waste of congruences and finally leads to a failure. Therefore, we choose the latter, which is different from \cite{ChenWang2024}.
	\end{remark}

	\begin{proof}[Proof of Theorem \ref{thm1}]
	We use $(a,m,r,p)_f$ and $(b,n,s,q)_l$ to denote $F_{mk+a}\equiv r\pmod{p}$ and $L_{nk+b}\equiv s\pmod{q}$, respectively.
	Thus,
	\begin{eqnarray*} &&\{ (a_i, m_i, r_i, p_i )_f \}_{i=1}^{33}\\
		&=& \{ (1,3,1,2)_f, (2,3,1,2)_f,   (3,8,2,3)_f, (5,8,2,3)_f, (6,8,2,3)_f,\\
		&&  (7,16,6,7)_f, (9,16,6,7)_f, (10,16,6,7)_f,(0,48,0,23)_f,(24,48,0,23)_f,\\
		&& (15, 32,46,47)_f,(17, 32,46,47)_f,(18, 32,46,47)_f, \\
		&&(60,96,959,1103)_f, (84,96,959,1103)_f, (9,18,15,19)_f,\\
		&&(24,36,9,17)_f,(30,36,9,17)_f,(33,192,2874,3167)_f,(159,192,2874,3167)_f,\\
		&&(1,64,1,2207)_f,(2,64,1,2207)_f,(63,64,1,2207)_f,(12,144,144,103681)_f,\\
		&&(36,288,14930352,10749957121)_f,(108,288,14930352,10749957121)_f,\\
		&&(546,576,115561578124837690841,115561578124838522881)_f,\\
		&&(100,128,680,1087)_f,(36,128,4141,4481)_f,\\
		&&(324,384,10314566492783, 11862575248703)_f,(162,1152,65243, 270143)_f,\\
		&&(516,1152,1548008755920,  25033626656641)_f,\\
		&&(738,1152, 784086245571237641757, 1974737795746080149567)_f\} .
	\end{eqnarray*}
	\begin{eqnarray*} &&\{ (b_j, n_j, s_j, q_j )_l \}_{j=1}^{15}\\
		&=& \{ (1,3,1,2)_l, (2,3,1,2)_l, (2,4,3,5)_l,  (0,8,2,3)_l, (5,8,2,3)_l, (7,8,2,3)_l,\\
		&&(9,16,6,7)_l,  (12,48,0,23)_l,(36,48,0,23)_l,(17, 32,46,47)_l, (15,18,15,19)_l,\\
		&&(27,36,9,17)_l,(3,72,4,107)_l,(1,64,1,2207)_l,(33,192,484,769)_l\} .
	\end{eqnarray*}
	We note that $$ m_i=\chi_f(p_i)\quad   \text{for}\quad   1\leq i\leq 33 \quad  \text{and} \quad  n_j=\chi_l(q_j) \quad  \text{for}\quad  1\leq j\leq 15. $$ Let 
	\begin{equation}\label{eq2}
		A=\{x\geq 0:x\equiv r_i\hskip -2.5mm\pmod{p_i} \ (1\le i\le 33),\quad x\equiv s_j\hskip -2.5mm\pmod{q_j} \ (1\le j\le 15)\}.\end{equation} 
	It follows from the Chinese remainder theorem that $$A=\{  M k+N : k=0,1,2, \ldots \} .$$
	
	On the one hand, we assume that $x\in A, $ $x=p+f, p\in \mathcal{P}, f\in \mathcal{F}.$ It follows from Lemma \ref{cs} that we can find $1\le i_0\le
		33$ such that $f=F_{m_{i_0}k+a_{i_0}}.$
	Thus,
	$$f=F_{m_{i_0}k+a_{i_0}}\equiv r_{i_0}\hskip -2.5mm\pmod{p_{i_0}}.$$
	By \eqref{eq2}, $x\equiv r_{i_0}\pmod{p_{i_0}}.$ Hence
	$$p=x-F_{m_{i_0}k+a_{i_0}}\equiv r_{i_0}-r_{i_0}\equiv 0\hskip -2.5mm\pmod{p_{i_0}}.$$
	It follows that $p=p_{i_0}.$ Therefore,
	\begin{equation}\label{eq4f}
		x=p_{i_0}+F_{m_{i_0}k+a_{i_0}}.
	\end{equation}
	By \eqref{eq2},
	$$x\equiv 9\hskip -2.5mm\pmod{17},\quad x\equiv 15\hskip -2.5mm\pmod{19},\quad x\equiv 46\hskip -2.5mm\pmod{47}.$$
	By \eqref{eq4f},
	$$F_{m_{i_0}k+a_{i_0}}\equiv 9-p_{i_0}\hskip -2.5mm\pmod{17},\   F_{m_{i_0}k+a_{i_0}}\equiv 15-p_{i_0}\hskip -2.5mm\pmod{19},\  F_{m_{i_0}k+a_{i_0}}\equiv 46-p_{i_0}\hskip -2.5mm\pmod{47}.$$
	
	For every $1\le i_0\le 33,$ we use the properties of $F_m$ to get a contradiction. For the convenience of the readers, we list tables $F_m\pmod{17},$ $F_m\pmod{19}$ and $F_m\pmod{47}.$
	\begin{table}[!ht]
		\centering\begin{tabular}{|c|c|c|c|c|c|c|c|c|c|c|c|c|}
			\hline $m\pmod{36}$  & $1$ & $2$ & $3$ & $4$ & $5$ & $6$ & $7$ &  $8$ & $9$   & $10$ & $11$ & $12$  \\
			\hline $F_m\pmod{17}$  & $1$ & $1$ & $2$ &  $3$ &  $5$ & $8$
			& $13$ & $4$ & $0$ & $4$ &  $4$ & $8$   \\
			\hline $m\pmod{36}$ & $13$ & $14$ & $15$ &  $16$ & $17$ & $18$ & $19$ & $20$ & $21$ & $22$ & $23$ & $24$  \\
			\hline $F_m\pmod{17}$  &  $12$ & $3$ & $15$ & $1$ & $16$ & $0$ &
			$16$ & $16$ &  $15$ &  $14$ & $12$ & $9$ \\
			\hline $m\pmod{36}$ & $25$ & $26$ & $27$ &  $28$ & $29$ & $30$ & $31$ & $32$ & $33$ & $34$ & $35$ & $36$  \\
			\hline $F_m\pmod{17}$  & $4$ & $13$ & $0$ & $13$ &
			$13$ &  $9$ &   $5$ & $14$ & $2$  & $16$ & $1$ & $0$  \\
			\hline
		\end{tabular}
		\caption{$F_m\pmod{17}$\label{tab17f}}
	\end{table}
	\begin{table}[!ht]
		\centering\begin{tabular}{|c|c|c|c|c|c|c|c|c|c|}
			\hline $m\pmod{18}$             & $1$ & $2$ & $3$ & $4$ & $5$ & $6$ &
			$7$ &  $8$ & $9$\\
			\hline $F_m\pmod{19}$   & $1$ & $1$ & $2$ &  $3$ &  $5$ & $8$
			& $13$ & $2$ & $15$\\
			\hline $m\pmod{18}$  & $10$ & $11$ & $12$ & $13$ & $14$ & $15$ &  $16$ & $17$ & $18$  \\
			\hline $F_m\pmod{19}$  & $17$ &  $13$ &
			$11$ &  $5$ &   $16$ & $2$ & $18$ & $1$ & $0$ \\
			\hline
		\end{tabular}
		\caption{$F_m\pmod{19}$\label{tab19f}}
	\end{table}
	\begin{table}[!ht]
		\centering\begin{tabular}{|c|c|c|c|c|c|c|c|c|c|c|c|c|c|c|c|c|}
			\hline $m\pmod{32}$             & $1$ & $2$ & $3$ & $4$ & $5$ & $6$ &
			$7$ &  $8$ & $9$ & $10$ & $11$ & $12$ & $13$ & $14$ & $15$ &  $16$ \\
			\hline $F_m\pmod{47}$  & $1$ & $1$ & $2$ &  $3$ &  $5$ & $8$
			& $13$ & $21$ & $34$  & $8$  & $42$   & $3$  & $45$  & $1$  & $46$  & $0$\\
			\hline $m\pmod{32}$  & $17$ & $18$ & $19$ & $20$ & $21$ & $22$ & $23$ & $24$ & $25$ & $26$ & $27$ &  $28$ & $29$ & $30$ & $31$ & $32$  \\
			\hline $F_m\pmod{47}$  & $46$ &  $46$ &
			$45$ &  $44$ &   $42$ & $39$ & $34$ & $26$ & $13$& $39$ &$5$ & $44$ & $2$ & $46$ & $1$ & $0$   \\
			\hline
		\end{tabular}
		\caption{$F_m\pmod{47}$\label{tab47f}}
	\end{table}
	
	\newpage
	If $m_{i_0}=3,$ then $p_{i_0}=2,$ $F_{3k+a_{i_0}}\equiv 9-p_{i_0}\equiv
	7\pmod{17}.$ By Table \ref{tab17f}, for any integer $m\geq 0,$ we have
	$F_m\not\equiv 7\pmod{17},$ a contradiction.
	
	If $m_{i_0}=8,$ then $p_{i_0}=3,$ $F_{8k+a_{i_0}}\equiv 12\pmod{19},$ a contradiction to Table \ref{tab19f}.
	
	If $m_{i_0}=16,$ then $p_{i_0}=7,$ $F_{16k+a_{i_0}}\equiv 2\pmod{17},$ where $ a_{i_0}\in\{7,9,10\}. $ It is clear that $ a_{i_0}=10 $ leads to a contradiction to Table \ref{tab17f}. Aware that $F_{16k+a_{i_0}}\equiv 39\pmod{47},$ where $ a_{i_0}\in\{7,9\}, $ a contradiction to Table \ref{tab47f}.
	
	If $m_{i_0}=18,$ then $p_{i_0}=19,$ $F_{18k+9}\equiv 27\pmod{47},$  a contradiction to Table \ref{tab47f}.
	
	If $m_{i_0}=32,$ then $p_{i_0}=47,$ $F_{32k+a_{i_0}}\equiv 6\pmod{19},$ a contradiction to Table \ref{tab19f}.
	
	If $m_{i_0}=36,$ then $p_{i_0}=17,$  $F_{36k+a_{i_0}}\equiv 17\pmod{19},$  where $ a_{i_0}\in\{24,30\}, $ a contradiction to Table \ref{tab19f}.
	
	If $m_{i_0}=48,$ then $p_{i_0}=23,$ $F_{48k+a_{i_0}}\equiv
	23\pmod{47},$ a contradiction to Table \ref{tab47f}.
	
	If $m_{i_0}=64,$ then $p_{i_0}=2207,$ $F_{64k+a_{i_0}}\equiv
	12\pmod{19},$ a contradiction to Table \ref{tab19f}.
	
	If $m_{i_0}=96,$ then $p_{i_0}=1103,$ $F_{96k+a_{i_0}}\equiv
	14\pmod{19},$ a contradiction to Table \ref{tab19f}.
	
	If $m_{i_0}=128,$ then $p_{i_0}\in\{1087,4481\},$ $F_{128k+a_{i_0}}\equiv
	40,30\pmod{47},$ a contradiction to Table \ref{tab47f}.
	
	If $m_{i_0}=144,$ then $p_{i_0}=103681,$ $F_{144k+12}\equiv
	11\pmod{17},$ a contradiction to Table \ref{tab17f}.
	
	If $m_{i_0}=192,$ then $p_{i_0}\in\{769,3167\},$ $F_{192k+a_{i_0}}\equiv
	29,28\pmod{47},$ a contradiction to Table \ref{tab47f}.
	
	If $m_{i_0}=288,$ then $p_{i_0}=10749957121,$ $F_{288k+a_{i_0}}\equiv
	11\pmod{17},$ a contradiction to Table \ref{tab17f}.
	
	If $m_{i_0}=384,$ then $p_{i_0}=11862575248703,$ $F_{384k+a_{i_0}}\equiv
	22\pmod{47},$ a contradiction to Table \ref{tab47f}.
	
	If $m_{i_0}=576,$ then $p_{i_0}=115561578124838522881$ and $F_{576k+a_{i_0}}\equiv
	11\pmod{17},$ a contradiction to Table \ref{tab17f}.
	
	If $m_{i_0}=1152,$ then $p_{i_0}\in\{270143,25033626656641,1974737795746080149567\}.$ For $p_{i_0}\in\{270143,25033626656641\},$ we have $F_{1152k+a_{i_0}}\equiv
	12,17\pmod{47},$ a contradiction to Table \ref{tab47f}. For $p_{i_0}=1974737795746080149567,$ we have $F_{1152k+738}\equiv
	13\pmod{19},$ a contradiction to Table \ref{tab19f}.
	
	On the other hand, we assume that $x\in A,$ $x=q+l, q\in \mathcal{P}, l\in \mathcal{L}.$
	It follows from Lemma \ref{cs} that we can find $1\le j_0\le
	15$ such that $l=L_{n_{j_0}k+b_{j_0}}.$ 
	Thus,
	$$l=L_{n_{j_0}k+b_{j_0}}\equiv s_{j_0}\hskip -2.5mm\pmod{q_{j_0}}.$$
	By \eqref{eq2}, $x\equiv s_{j_0}\pmod{q_{j_0}}.$ Hence
	$$q=x-L_{n_{j_0}k+b_{j_0}}\equiv s_{j_0}-s_{j_0}\equiv 0\hskip -2.5mm\pmod{q_{j_0}}.$$
	It follows that $q=q_{j_0}.$ Therefore,
	\begin{equation}\label{eq4}
		x=q_{j_0}+L_{n_{j_0}k+b_{j_0}}.
	\end{equation}
	By \eqref{eq2},
	$$x\equiv 15\hskip -2.5mm\pmod{19},\quad x\equiv 46\hskip -2.5mm\pmod{47}.$$
	By \eqref{eq4},
	$$L_{n_{j_0}k+b_{j_0}}\equiv 15-q_{j_0}\hskip -2.5mm\pmod{19},\quad L_{n_{j_0}k+b_{j_0}}\equiv 46-q_{j_0}\hskip -2.5mm\pmod{47}.$$
	
	For every $1\le j_0\le 15,$ we use the properties of $L_n$ to get a contradiction. For the convenience of the readers, we list $L_n\pmod{19}$ and $L_n\pmod{47}.$ 
	\begin{table}[!ht]
		\centering\begin{tabular}{|c|c|c|c|c|c|c|c|c|c|}
			\hline $n\pmod{18}$  & $1$ & $2$ & $3$ & $4$ & $5$ & $6$ &
			$7$ &  $8$ & $9$\\
			\hline $L_n\pmod{19}$   & $1$ & $3$ & $4$ &  $7$ &  $11$ & $18$
			& $10$ & $9$ & $0$\\
			\hline $n\pmod{18}$  & $10$ & $11$ & $12$ & $13$ & $14$ & $15$ &  $16$ & $17$ & $18$  \\
			\hline $L_n\pmod{19}$  & $9$ &  $9$ &
			$18$ &  $8$ &   $7$ & $15$ & $3$ & $18$ & $2$ \\
			\hline
		\end{tabular}
		\caption{$L_n\pmod{19}$\label{tab19}}
	\end{table}
	\begin{table}[!ht]
		\centering\begin{tabular}{|c|c|c|c|c|c|c|c|c|c|c|c|}
			\hline $n\pmod{32}$  & $1$ & $2$ & $3$ & $4$ & $5$ & $6$ &
			$7$ &  $8$ & $9$ & $10$ & $11$\\
			\hline $L_n\pmod{47}$   & $1$ & $3$ & $4$ &  $7$ & $11$ 
			& $18$ & $29$ & $0$ & $29$ & $29$ & $11$\\ 
			\hline $n\pmod{32}$  & $12$ & $13$ & $14$ & $15$ &  $16$ & $17$ & $18$ & $19$ & $20$ & $21$ & $22$ \\
			\hline $L_n\pmod{47}$    & $40$  & $4$  & $44$  & $1$  & $45$ & $46$ &  $44$ & $43$ &  $40$ & $36$ & $29$\\
			\hline $n\pmod{32}$    & $23$ & $24$ & $25$ & $26$ & $27$ &  $28$ & $29$ & $30$ & $31$ & $32$  &\\
			\hline $L_n\pmod{47}$    & $18$ & $0$ & $18$& $18$ &$36$ & $7$ & $43$ & $3$ & $46$ & $2$   &\\
			\hline
		\end{tabular}
		\caption{$L_n\pmod{47}$\label{tab47}}
	\end{table}
	\newpage
	
	If $n_{j_0}=3,$ then $q_{j_0}=2,$ $L_{3k+b_{j_0}}\equiv 15-q_{j_0}\equiv
	13\pmod{19}.$ By Table \ref{tab19}, for any integer $n\geq 0,$ we have
	$L_n\not\equiv 13\pmod{19},$ a contradiction.
	
	If $n_{j_0}=4,$ then $q_{j_0}=5,$ $L_{4k+2}\equiv 41\pmod{47},$ a contradiction to Table \ref{tab47}.
	
	If $n_{j_0}=8,$ then $q_{j_0}=3,$ $L_{8k+b_{j_0}}\equiv 12\pmod{19},$ a contradiction to Table \ref{tab19}.
	
	If $n_{j_0}=16,$ then $q_{j_0}=7,$ $L_{16k+9}\equiv 39\pmod{47},$ a contradiction to Table \ref{tab47}.
	
	If $n_{j_0}=18,$ then $q_{j_0}=19,$ $L_{18k+15}\equiv 27\pmod{47},$  a contradiction to Table \ref{tab47}.
	
	If $n_{j_0}=32,$ then $q_{j_0}=47,$ $L_{32k+17}\equiv 6\pmod{19},$ a contradiction to Table \ref{tab19}.
	
	If $n_{j_0}=36,$ then $q_{j_0}=17,$  $L_{36k+27}\equiv 17\pmod{19},$  a contradiction to Table \ref{tab19}.
	
	If $n_{j_0}=48,$ then $q_{j_0}=23,$ $L_{48k+b_{j_0}}\equiv
	23\pmod{47},$ a contradiction to Table \ref{tab47}.
	
	If $n_{j_0}=64,$ then $q_{j_0}=2207,$ $L_{64k+1}\equiv
	12\pmod{19},$ a contradiction to Table \ref{tab19}.
	
	If $n_{j_0}=72,$ then $q_{j_0}=107,$ $L_{72k+3}\equiv 33\pmod{47},$ a contradiction to Table \ref{tab47}.
	
	If $n_{j_0}=192,$ then $q_{j_0}=769,$ $L_{192k+33}\equiv
	6\pmod{19},$ a contradiction to Table \ref{tab19}.
	
	This completes the proof of Theorem \ref{thm1}.
\end{proof}

\section{Corollaries}

Let $$ B_f=\{ n>1 : r_f(n)=0\},\quad B_l=\{ n>1 : r_l(n)=0\}. $$
By the Chinese remainder theorem, the following two corollaries can be deduced from the proof of Theorem \ref{thm1}.
\begin{corollary}
	Let
	\begin{eqnarray*} M_f&=&1123655508683096233894389695493505447961799048381240628277073218254640\\
		&&1792062598881141469403328411757364284878802,\\
		N_f&=&2468127909106617673449389274478768677072124807451258375411168391979806\\
		&&255191037873183154992460933770371093595473,
	\end{eqnarray*}
	%$$A_f=\{  M_f k+N_f : k=0,1,2, \ldots \} ,$$
	then $ M_f k+N_f\in B_f $ for every integer $k\geq 0.$
	%then $ A_f\subseteq B_f. $
\end{corollary}

\begin{corollary}
	Let
	\begin{eqnarray*} M_l&=&578938092213810,\\
		N_l&=&85206628521871,
	\end{eqnarray*}
	%$$A_l=\{  23(M_l k+N_l) : k=0,1,2, \ldots \} ,$$
	then $ 23(M_l k+N_l)\in B_l $ for every integer $k\geq 0.$
	%then $ A_l\subseteq B_l. $
\end{corollary}

Recall the covering system 
\begin{eqnarray*} 
	&&\{ (1,3), (2,3), (2,4),  (0,8), (5,8), (7,8),(9,16),(12,48),(36,48),\\
	&&(17, 32), (15,18),(27,36),(3,72),(33,96)\},
\end{eqnarray*}	
which is mentioned in Remark \ref{remark} for Lucas numbers. 
We note that
$$L_{3k+1}\equiv L_{3k+2}\equiv 1\hskip -2.5mm\pmod{2}.$$
By using $$ L_{96k+33}\equiv 261\hskip -2.5mm\pmod{1103} $$ instead of $$L_{64k+1}\equiv 1\hskip -2.5mm\pmod{2207}\quad\text{and}\quad L_{192k+33}\equiv 484\hskip -2.5mm\pmod{769},$$ the following corollary can be deduced from a similar proof of Theorem \ref{thm1}.
\begin{corollary}
	Let
	\begin{eqnarray*} M_l'&=&8653798948830,\\
		N_l'&=&3695757248273,
	\end{eqnarray*}
	%$$A_l'=\{  M_l' k+N_l' : k=0,1,2, \ldots \} ,$$
	then $ M_l' k+N_l'\in B_l $ for every integer $k\geq 0.$
	%then $ A_l'\subseteq B_l. $
\end{corollary}
\section*{Acknowledgments}

%The work is supported by the Project of Graduate Education Innovation of Jiangsu Province, Grant No. KYCX22\_1532.
%The author would like to thank Professor Yong-Gao Chen for the helpful suggestions.

The author is supported by the Natural Science Foundation of the Jiangsu Higher Education Institutions of China (Grant No. 24KJD110001).

\end{document}